\documentclass{amsart}
\newtheorem{Theo}{Theorem}
\newtheorem{Lem}{Lemma}
\begin{document}
\title[Lower Bounds for the Large Sieve]{Lower Bounds for Expressions
  of Large Sieve Type} 
\author[J.-C. Schlage-Puchta]{Jan-Christoph Schlage-Puchta}
\begin{abstract}
We show that the large sieve is optimal for almost all exponential
sums. 
\end{abstract}
\maketitle
MSC-Index: 11N35, 30B20\\[2mm]
Let $a_n, 1\leq n\leq N$ be complex numbers, and set
$S(\alpha)=\sum_{n\leq N} a_n e(n\alpha)$, where $e(\alpha)=\exp(2\pi i\alpha)$. Large Sieve inequalities
aim at bounding the number of places where this sum can be
extraordinarily large, the basic one being the bound 
\[
\sum_{q\leq Q}\underset{(a, q)=1}{\sum_{1\leq a\leq q}} \left|S\left(\frac{a}{q}\right)\right|^2
\leq (N+Q^2)\sum_{n\leq N}|a_n|^2
\]
(see e.g. \cite{Mon} for variations and applications). P. Erd\H os and
A. R\'enyi \cite{ER} considered lower bounds of the same type, 
in particular they showed that the bound
\begin{equation}\label{ER}
\sum_{q\leq Q}\sum_{(a, q)=1} \left|S\left(\frac{a}{q}\right)\right|^2
\ll N\sum_{n\leq N}|a_n|^2,
\end{equation}
valid for $Q\ll\sqrt{N}$, is wrong for almost all choices of
coefficients $a_n\in\{1, -1\}$, provided that $Q>C\sqrt{N}\log N$, and
that the standard probabilistic argument fails to decide whether
(\ref{ER}) is true in the range $\sqrt{N}<Q<\sqrt{N}\log N$. In this
note, we show that (\ref{ER}) indeed fails throughout this
range.

\begin{Theo}
Let $S(\alpha)$ be as above. Then
\begin{equation}\label{probab}
\sum_{q\leq Q}\sum_{(a, q)=1}\left|S\left(\frac{a}{q}\right)\right|^2
\geq \varepsilon Q^2\sum_{n\leq N} |a_n|^2
\end{equation}
holds true with probability tending to 1 provided $\varepsilon$ tends to
0, and $Q^2/N$ tends to infinity.
\end{Theo}

Our approach differs from \cite{ER} in so far as we first prove an
unconditional lower bound, which involves an awkward expression, and
show then that almost always this expression is small. We show the following.

\begin{Lem}
Let $S(\alpha)$ be as above, and define
\[
M(x)=\sup\limits_\mathfrak{m} \frac{\int\limits_\mathfrak{m} |S(u)|^2\; du}{\int\limits_0^1
|S(u)|^2\; du},
\]
where $\mathfrak{m}$ ranges over all measurable subsets of $[0, 1]$ of measure
$x$. Then for any real parameter $A>1$ we have the estimate
\begin{equation}\label{sieve}
\sum_{q\leq Q}\sum_{(a, q)=1}\left|S\left(\frac{a}{q}\right)\right|^2
\geq \left(\frac{Q^2}{A}\left(1-M\left(\frac{1}{A}\right)\right) - 6\pi
N A \right) \sum_{n\leq N} |a_n|^2.
\end{equation}
\end{Lem}

\begin{proof}
Our proof adapts Gallagher's proof of an upper bound large
sieve \cite{Gal}. For every $f\in C^1([0, 1])$, we have 
\[
f(1/2) = \int\limits_0^1 f(u)\;du + \int\limits_0^{1/2} uf'(u)\;du -
\int\limits_{1/2}^1 (1-u)f'(u)\;du.
\]
Putting
$f(u)=|S(u)|^2$, and using the linear substitution $u\mapsto
(\alpha-\delta/2)+\delta u$, we obtain for every $\delta>0$
and any $\alpha\in[0, 1]$
\begin{eqnarray*}
|S(\alpha)|^2 & = & \frac{1}{\delta}\int\limits_{\alpha-\delta/2}^{\alpha+\delta/2}
|S(u)|^2\;du + \frac{1}{\delta}\int\limits_{\alpha-\delta/2}^{\alpha}
\big(\delta/2-|u-\alpha|\big)\big(S'(u)S(-u)-S(u)S'(-u)\big)\; du \\
 &&\qquad - \frac{1}{\delta}\int\limits_{\alpha}^{\alpha+\delta/2}
\big(\delta/2-|u-\alpha|\big)\big(S'(u)S(-u)-S(u)S'(-u)\big)\; du.
\end{eqnarray*}
We have $|S(u)|=|S(-u)|$ and $|S'(-u)|=|S'(u)|$, thus $|S'(u)S(-u)-S(u)S'(-u)|
\leq 2|S(u)S'(u)|$, and we obtain
\begin{eqnarray*}
|S(\alpha)|^2 & \geq & \frac{1}{\delta}\int\limits_{\alpha-\delta/2}^{\alpha+\delta/2}
|S(u)|^2\;du - \frac{1}{\delta}\int\limits_{\alpha-\delta/2}^{\alpha+\delta/2}2\Big(\frac{1}{2}-\frac{|u-\alpha|}{\delta}\Big)|S(u)S'(u)|\; du.\\
 & \geq & \frac{1}{\delta}\int\limits_{\alpha-\delta/2}^{\alpha+\delta/2}
|S(u)|^2\;du - \int\limits_{\alpha-\delta/2}^{\alpha+\delta/2}|S(u)S'(u)|\; du.
\end{eqnarray*}
We now set $\delta = A/Q^2$. We can safely assume that
$\delta<\frac{1}{2}$, since our claim would be trivial
otherwise. Summing over all fractions $\alpha = \frac{a}{q}$ with
$q\leq Q, (a, q)=1$, we get 
\begin{multline}\label{long}
\sum_{q\leq Q}\,\sum_{(a, q)=1}\left|S\left(\frac{a}{q}\right)\right|^2
\;\geq\; \frac{Q^2}{A}\int\limits_0^1|S(u)|^2\;du\;-\\
\frac{Q^2}{A}\int\limits_{m(Q, A)}|S(u)|^2\;du\; -\;
\int\limits_0^1R(u)|S(u)S'(u)|\; du,
\end{multline}
where
\[
R(u) = \#\left\{a, q: (a, q)=1, q\leq Q,\left|u-\frac{a}{q}\right|\leq
\frac{A}{Q^2}\right\},
\]
and
\[
m(Q, A) = \{u\in[0, 1]: R(u)=0\}.
\]
To bound $R(u)$, let $\frac{a_1}{q_1}<\frac{a_2}{q_2}< \cdots <
\frac{a_k}{q_k}$ be the list of all fractions with $q_i\leq Q$,
$\left|u-\frac{a_i}{q_i}\right|\leq \frac{A}{Q^2}$. We have for $i\neq
j$ the bound
\[
\left|\frac{a_i}{q_i}-\frac{a_j}{q_j}\right|\geq\frac{1}{q_iq_j}\geq\frac{1}{Q^2},
\]
that is, the fractions $\frac{a_1}{q_1},\ldots, \frac{a_k}{q_k}$ form
a set of points with distance $>\frac{1}{Q^2}$ in an interval of
length $\frac{2A}{Q^2}$. There can be at most $2A+1$ such
points, hence, $R(u)\leq 3A$.

Next, we bound $|m(Q, A)|$. By Dirichlet's theorem, we have that for each
real number $\alpha\in[0, 1]$ there exists some $q\leq Q$ and some
$a$, such that $|\alpha-\frac{a}{q}|\leq\frac{1}{qQ}$. If
$\alpha\in m(Q, A)$, we must have $\frac{1}{qQ}>\frac{A}{Q^2}$, that
is, $q<Q/A$. Hence, we obtain
\begin{eqnarray*}
|m(Q, A)| & \leq & \left|\bigcup_{q<Q/A}\bigcup_{(a, q)=1}
\left[\frac{a}{q}-\frac{1}{qQ}, \frac{a}{q}+\frac{1}{qQ}\right]\setminus\left[\frac{a}{q}-\frac{A}{Q^2}, \frac{a}{q}+\frac{A}{Q^2}\right]\right|\\
 & \leq & \sum_{q<Q/A} \frac{\varphi(q)(2Q-2Aq)}{qQ^2} \leq
\frac{1}{Q^2}\int_0^{Q/A} 2Q-2At\;dt = \frac{1}{A}.
\end{eqnarray*}
We can now estimate the right hand side of (\ref{long}). The first
summand is $\frac{Q^2}{A}\sum_{n\leq N} |a_n|^2$, while the second is
by definition at most $\frac{Q^2}{A}M(1/A)$. For the third we apply the
Cauchy-Schwarz-inequality to obtain
\begin{eqnarray*}
\left(\int\limits_0^1|S(u)S'(u)|\; du\right)^2 & \leq &
\left(\int\limits_0^1|S(u)|^2\; du\right)\left(\int\limits_0^1
|S'(u)|^2\; du\right)\\ 
 & = & \left(\sum_{n\leq N} |a_n^2|\right)\left(\sum_{n\leq N} (2\pi
n)^2|a_n^2|\right)\\
 & \leq & (2\pi N)^2\left(\sum_{n\leq N} |a_n^2|\right)^2.
\end{eqnarray*}
Hence, the last term in (\ref{long}) is bounded above by $3A(2\pi
N)\sum_{n\leq N}|a_n|^2$, and inserting our bounds into (\ref{long})
yields the claim of our lemma.
\end{proof}

Now we deduce Theorem~1. Let $S(\alpha)$ be a random sum in the sense
that the coefficients 
$a_n\in\{1, -1\}$ are chosen at random. We compute the expectation of
the fourth moment of $S(\alpha)$.
\begin{eqnarray*}
\mathrm{E}\int\limits_0^1|S(u)|^4\;du & = & \mathrm{E}
\sum_{{\mu_1+\mu_2=\nu_1+\nu_2}\atop{\mu_1, \mu_2, \nu_1,
\nu_2\leq N}} a_{\nu_1}a_{\nu_2}a_{\mu_1}a_{\mu_2}\\
 & = & \#\{\mu_1, \mu_2, \nu_1, \nu_2\leq N: \{\mu_1, \mu_2\} =
 \{\nu_1, \nu_2\}\}\\
 & = & 2N^2-N.
\end{eqnarray*}
If $m\subseteq[0,1]$ is of measure $x$, then $\int\limits_m|S(u)|^2\;du
\leq \sqrt{x}\left(\int\limits_m|S(u)|^4\;du\right)^{1/2}$, thus
$\mathrm{E} M(x)\leq \sqrt{2x}$. In particular, we have $M(x)\leq 1/2$
with probability $\geq 1-\sqrt{8x}$. Let $\delta>0$ be given, and set
$A=8\delta^{-2}$. Then with probability $\geq 1-\delta$ we have
$M(1/A)\leq 1/2$, and (\ref{sieve}) becomes
\begin{eqnarray*}
\sum_{q\leq Q}\sum_{(a, q)=1}\left|S\left(\frac{a}{q}\right)\right|^2
 & \geq & \left(\frac{Q^2\delta^2}{16} - 48\delta^{-2}\pi
N \right) \sum_{n\leq N} |a_n|^2\\
 & \geq &\frac{Q^2\delta^2}{32} \sum_{n\leq N} |a_n|^2,
\end{eqnarray*}
provided that $Q^2>1536\delta^4 N$. Hence, for fixed $\epsilon$, the
relation (\ref{probab}) becomes true with probability
$1-\sqrt{1024\epsilon}$, provided that $Q^2/N$ is sufficiently
large. Hence, our claim follows.

I would like to thank the referee for improving the quality of this paper.

Jan-Christoph Schlage-Puchta\\
Department of Pure Mathematics and Computer Algebra\\
Universiteit Gent\\
Krijgslaan 281\\
9000 Gent\\
Belgium\\
\verb$jcsp@cage.ugent.be$


\begin{thebibliography}{9}
\bibitem{ER} P. Erd\H os, A. R\'enyi, Some remarks on the large
sieve of Yu. V. Linnik, {\em Ann. Univ. Sci. Budapest. E\"otv\"os
Sect. Math.} {\bf 11} (1968), 3--13.
\bibitem{Gal} P. X. Gallagher, The large sieve, {\em Mathematika} {\bf
    14} (1967), 14--20. 
\bibitem{Mon} H. L. Montgomery, {\em Multiplicative Number
Theory}, Lecture Notes in Mathematics, Vol. 227 (1971).
\end{thebibliography}
\end{document}